\documentclass[12pt]{amsart}

\usepackage{amssymb}

\makeatletter
\@addtoreset{equation}{section}
\makeatother

\theoremstyle{plain}
\newtheorem{theorem}{Theorem}[section]
\newtheorem{proposition}[theorem]{Proposition}

\theoremstyle{definition}

\newcommand{\Cee}{{\mathbb C}}
 
\newcommand{\fB}{\mathcal{B}}

\newcommand{\fH}{\mathcal{H}}
\newcommand{\fM}{\mathcal{M}}
\newcommand{\fU}{\mathcal{U}}
 
\begin{document}
 
 
\author{Nico Spronk}
\address{Department of Pure Mathematics, University of Waterloo, Waterloo, Ontario, N2L3G1,Canada}
\email{nspronk@uwaterloo.ca}

\title[Hulanicki's Theorem]{A short proof of Hulanicki's Theorem}

\begin{abstract}
We outline a simple proof of Hulanicki's theorem, that a locally compact group is amenable
if and only if the left regular representation weakly contains all unitary representations.
This combines some elements of the literature which have not appeared together, before.
\end{abstract}



\thanks{This research was partially supported by an NSERC Discovery Grant.}


 \date{\today}
 
 \maketitle
 

Let $G$ be a locally compact group.  For $p=1,2,\infty$ we let $L^p(G)$ denote the $L^p$-space with respect 
to the left Haar measure $m$ ($dx=dm(x)$).  For a unitary representation $\pi:\fU(\fH)$ (we assume 
continuous with respect to weak or strong operator topology on unitary operators $\fU(\fH)$), the integrated 
form is given by $\pi:L^1(G)\to\fB(\fH)$, $\pi(f)=\int_G f(x)\pi(x)\,dx$ (integral understood in weak operator 
sense), and is well-known to be contractive and satisfy $\pi(f^*)=\pi(f)^*$.
The left regular representation $\lambda:G\to\fU(L^2(G))$ is
given by $\lambda(x)f=x\ast f$, where $x\ast f(y)=f(x^{-1}y)$ for $m$-a.e.\ $y$ in $G$.  Its integrated
form is given by $\lambda:L^1(G)\to\fB(L^2(G))$, $\lambda(f)h=f\ast h$.

We recall that $G$ is {\it amenable} if $L^\infty(G)$ admits a left invariant mean, a linear functional $\mu$
such that $\mu(\varphi)\geq 0$ provided $\varphi\geq 0$, $\mu(1)=1$, and $\mu(x\ast\varphi)=\mu(\varphi)$ 
for all $x$ in $G$ and $\varphi$ in $L^\infty(G)$.   This is well-known to be equivalent to 
Reiter's condition (P$_1$), i.e.\ the existence of a {\it Reiter net} $(r_\alpha)$  in $L^1(G)$: 
each $r_\alpha\geq 0$ $m$-a.e., $\int_Gr_\alpha\,dm=1$ and $\lim_\alpha\|x\ast r_\alpha-r_\alpha\|_1=0$ 
uniformly for $x$ in compact sets.  See \S3.2 of the classic book of Greenleaf \cite{greenleaf}.

\begin{theorem}\label{theo:hulanicki}
{\rm (Hulanicki \cite{hulanicki,hulanicki1})}  A locally compact group $G$ is amenable if and only if
for any unitary representation $\pi:G\to\fU(\fH)$, we have $\|\lambda(f)\|\geq \|\pi(f)\|$ for
all $f$ in $L^1(G)$.
\end{theorem}

The necessity condition above is the property that $\lambda$ {\it weakly contains} all
unitary representations.  This is equaivalent to having, for every $\pi$, the existence of a representation
$\pi_\lambda:C^*_\lambda\to C^*_\pi$ for which $\pi_\lambda\circ\lambda=\pi$ on $L^1(G)$,
where $C^*_\pi=\overline{\pi(L^1(G))}$ (norm closure in $\fB(\fH)$).  None of the elements of
this direction of the proof are novel, but are combined in a manner which does not seem to appear
in the existant literature.

\begin{proof}[Proof of necessity]
If $(r_\alpha)$ is a Reiter net in $L^1(G)$,  
then $k_\alpha=r_\alpha^{1/2}$ defines a net in $L^2(G)$ such that
matrix coefficients $\langle \pi(\cdot)k_\alpha,k_\alpha\rangle$ tend uniformly on compacta to $1$.
(This is the easy direction of Theorem 3.5.2 in \cite{greenleaf}.)  Since
compactly supported elements are dense in $L^2(G)$, we conclude that amenablity of $G$ is
entails having a net $(u_\alpha)$ of compactly supported positive definite elements
which converges uniformly of compacta to $1$.  A theorem of Godement \cite{godement}
(see 13.8.6 the book of Dixmier \cite{dixmier}) shows that any compactly supported positive definite
function on $G$ is of the form $\langle \pi(\cdot)h,h\rangle$ for some $h$ in $L^2(G)$. 

Let $f\in L^1(G)$.  By density of such elements, we may assume that $f$ is compactly supported. 
Given a unitary representation $\pi$ and $\varepsilon>0$, let $\xi$ be a unit vector in $\fH$ for which
$\|\pi(f)\xi\|^2+\varepsilon>\|\pi(f)\|^2$.  Then given the net $(u_\alpha)$, promised above, we find
elements $h_\alpha$ in $L^2(G)$ for which
\[
\langle \lambda(\cdot)h_\alpha,h_\alpha\rangle=\langle\pi(\cdot)\xi,\xi\rangle u_\alpha
\overset{\alpha}{\longrightarrow}\langle\pi(\cdot)\xi,\xi\rangle\text{ uniformly on compacta}
\]
so we compute
\begin{align*}
\|\pi(f)\|^2-\varepsilon&<\int_G f^*\ast f(x)\langle\pi(x)\xi,\xi\rangle\,dx \\
&=\lim_\alpha \int_G f^*\ast f(x)\langle\pi(x)\xi,\xi\rangle u_\alpha(x)\,dx \\
&=\lim_\alpha  \int_G f^*\ast f(x)\langle \lambda(x)h_\alpha,h_\alpha\rangle\,dx \\
&\leq\lim_\alpha \|\lambda(f)\|^2\langle\lambda(e)h_\alpha,h_\alpha\rangle=\|\lambda(f)\|^2
\end{align*}
which establishes the desired inequality.
\end{proof}

To prove the sufficiency condition, we shall use a specialization of a {\it multiplicative domain}
result of Choi \cite{choi} to states.  We recall that a state on a unital C*-algebra $\fB$ is any
functional $\tau$ which satisfies $\tau(B^*B)\geq 0$ for $B$ in $\fB$ and $\tau(I)=1$.

\begin{proposition}\label{prop:choi}
Let $\fB$ be a unital C*-algebra and $\fM$ a unital C*-subalgebra of $\fB$.  Suppose a state
$\tau$ on $\fB$ satisfies $\tau(A^*A)=|\tau(A)|^2$ for $A$ in $\fM$.  Then $\tau$ satisfies
$\tau(AB)=\tau(A)\tau(B)=\tau(BA)$ for $A$ in $\fM$ and $B$ in $\fB$.
\end{proposition}

\begin{proof}
Let $(\fH,\pi,\xi)$ be the Gelfand-Naimark-Segal triple associated with $\tau$,
\i.e.\ $\tau=\langle\pi(\cdot)\xi,\xi\rangle$.
Let $P=\langle \cdot,\xi\rangle\xi$, so $P$ is the othogonal projection onto
$\Cee\xi$, and $P\pi(B)\xi=\tau(B)\xi$ for $B$ in $\fB$.  Then, if $A\in\fM$ we have
\[
\|(I-P)\pi(A)\xi\|^2=\langle (I-P)\pi(A)\xi|\pi(A)\xi\rangle=\tau(A^*A)-\tau(A)\overline{\tau(A)}=0.
\]
Hence $\Cee\xi$ is $\pi(\fM)$-invariant, and it follows that $\pi(A)\xi=\tau(A)\xi$ for $A$ in $\fM$.
Hence, if $B\in\fB$ we have
\[
\tau(AB)=\langle \pi(B)\xi,\pi(A^*)\xi\rangle =\tau(A)\tau(B)
\]
and, similarly, $\tau(BA)=\tau(A)\tau(B)$.
\end{proof}

The technique below has been observed for discrete groups in the book of Brown and 
Ozawa \cite{browno}.  The author is unsure of the origin of this trick for the purposes of this theorem.

\begin{proof}[Proof of sufficiency condition of Theorem \ref{theo:hulanicki}]
Let $M(G)$ be the measure algebra of $G$, in which $L^1(G)$ is the ideal of elements
absolutely continuous with respect to $m$.  Consider the augmentation characters
\begin{align*}
&\alpha:L^1(G)\to\Cee,\; \alpha(f)=\int_G f\,dm \\ 
&\tilde{\alpha}:M(G)\to\Cee,\;\tilde{\alpha}(\mu)=\mu(G)=\int_G 1\,d\mu
\end{align*}
which are $*$-homomorphisms with $\tilde{\alpha}|_{L^1(G)}=\alpha$.  Hence, our assumptions entail that
\[
|\alpha(f)|\leq\|\lambda(f)\|\text{ for }f\text{ in }L^1(G).
\]
Now, fix $f$ in $L^1(G)$ with $f\geq 0$ $m$-a.e.\ and $\alpha(f)=1$, so $\|f\|_1=1$.  
Then if $\mu\in M(G)$ we have
\[
|\tilde{\alpha}(\mu)|=|\alpha(\mu\ast f)|\leq \|\lambda(\mu\ast f)\|\leq\|\lambda(\mu)\|
\]
where $\lambda(\mu)h=\mu\ast h$ for $h$ in $L^2(G)$.  Hence $\tilde{\alpha}$ extends
to a multiplicative functional $\tau$ on $M_\lambda^*=\overline{\lambda(M(G))}$, satisfying
$\tau(\lambda(\mu))=\tilde{\alpha}(\mu)$.  This is clearly a state which satisfies
$\tau(A^*A)=|\tau(A)|^2$ on $M_\lambda^*$.
Let $\tilde{\tau}$ denote any norm perserving extension to $\fB(L^2(G))$.

We let $M:L^\infty(G)\to\fB(L^2(G))$ denote the $*$-homomorphism into multiplication operators:
$M(\varphi)h=\varphi h$.  It is well known, and standard to verify, that $\lambda(x)M(\varphi)\lambda(x^{-1})
=M(x\ast\varphi)$.  Since $\lambda(x)=\lambda(\delta_x)$ (Dirac measure at $x$), we see 
from Proposition \ref{prop:choi} that
\[
\tilde{\tau}\circ M(x\ast\varphi)=\tilde{\tau}(\lambda(\delta_x)M(\varphi)\lambda(\delta_{x^{-1}}))
=\tilde{\alpha}(\delta_x)\tilde{\tau}\circ M(\varphi)\tilde{\alpha}(_{x^{-1}})=\tilde{\tau}\circ M(\varphi).
\]
Hence $\mu=\tilde{\tau}\circ M$ is a left invariant mean on $L^\infty(G)$.
\end{proof}

The only claim of originality made by the author is the transferring of the multiplicative domain
technique of Theorem 2.6.8 of \cite{browno} to $M_\lambda^*$ in the proof of sufficiency.  



\end{document}